\documentclass{amsart}
\usepackage[dvipdfmx]{graphicx}
\usepackage{amscd}
\usepackage{color}
\usepackage{amsmath,amssymb,amsthm}

\newtheorem{Def}{Definition}[section]
\newtheorem{thm}{Theorem}[section]

\newtheorem{lem}{Lemma}[section]
\newtheorem{prop}{Proposition}[section]
\newtheorem{claim}{Claim}[section]
\newtheorem{Rem}{Remark}[section]

\begin{document}

\title[Stable equivalence of a handlebody decomposition]{Stable equivalence of handlebody decompositions whose partitions are multibranched surfaces}
\author{Masaki Ogawa}
\date{}
\begin{abstract}
In this paper, we consider decompositions of closed orientable 3-manifolds with more than 3 handlebodies, where the union of intersections of handlebodies is a multibranched surface. We define stabilization operations for such decompositions and show the stable equivalence.
\end{abstract}

\maketitle
\section{Introduction}
Recently, we introduced a handlebody decomposition of a closed orientable 3-manifold \cite{IM, Poly}. 
If a 3-manifold is decomposed into several handlebodies, then we call this decomposition a handlebody decomposition.
In such decomposition, we call a union of intersection of handlebodies a partition.
In particular, if the partition of handlebody decomposition is simple polyhedron, we say handlebody decomposition is simple.
This is a generalization of a Heegaard splitting and a trisection of a 3-manifold. 
It is well known that Heegaard splittings of the same 3-manifold are stable equivalent \cite{R, S}.
Recently,  Koenig showed a stable equivalence of a trisection of a 3-manifold in \cite{K}.
Also, in \cite{Poly}, we showed the stable equivalence of simple handlebody decompositions.
In \cite{Poly}, we consider the case where handlebody decomposition is simple. To show the stable equivalence of such handlebody decomposition, we use not only one but some types of stabilizations and moves on a simple polyhedron. Stabilizations used in \cite{Poly} are called  type 0 and type 1 stabilizations.  A type 0 stabilization is similar to a stabilization of a Heegaard splitting. 

A multibranched surface is a 2-dimensional complex such that a point in it have a regular neighborhood homeomorphic to a disk or branched point where a branched point is a point whose neighborhood as in Figure \ref{mulsurf_1}. 
\begin{figure}[h]
					\centering
					\includegraphics[scale=0.5]{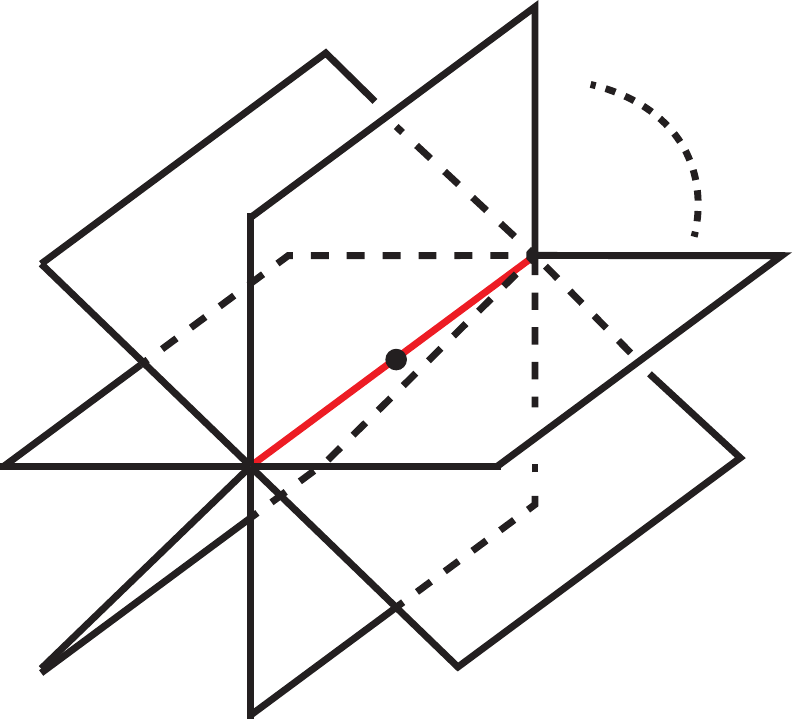}
					\caption
					{A regular neighborhood of a branched point in a multibranched surface.}
					\label{mulsurf_1}
\end{figure}
We consider the multibranched surface embedded in a 3-manifold which separates it into some handlebodies. For a multibranched surface in a 3-manifold, Ishihara, Koda, Ozawa and Shimokawa introduced IX and XI moves which do not change the regular neighborhood \cite{KOKK}. See Section 2 for the detail.

If a 3-manifold is decomposed into some handlebodies so that the union of intersections of handlebodies is a multibranched surface, we call this a multibranched handlebody decomposition.
If a multibranched handlebody decomposition consists of three handlebodies, this corresponds to a handlebody decomposition with no vertex in the partition.  On the other hand, if the number of handlebodies is greater than 4, there exists a multibranched handlebody decomposition which does not correspond to simple handlebody decomposition. 
We say two multibranched handlebody decompositions are isotopic if the union of the intersections of handlebodies is isotopic to each other.
For  such  decomposition, we show the following theorem.
\begin{thm}\label{stbthm}
	Two multibranched handlebody decompositions with four handlebodies of the same 3-manifold are isotopic to each other after applying XI  and IX moves and type 0, 1 stabilizations finitely many times. 
\end{thm}

Also, we characterize 3-manifolds which have a certain multibranched handlebody decomposition.
We say that such multibranched handlebody decomposition has a type-$(g_1, . . . , g_n)$ decomposition if a handlebody $H_i$ of the decomposition has genus $g_i$ for $i=1, . . . , n$.
We  consider 3-manifolds which have a handlebody decomposition with exactly four handlebodies with small genera.
In this paper, lens space is a 3-manifold with genus one Heegaard splitting which is not homeomorphic to both the 3-sphere and  $S^2\times S^1$.
Let $\mathbb B$ be a connected sum of a finite number of  $S^2\times S^1$'s,  let $\mathbb L$ and $\mathbb L_i$ be lens spaces, and let $\mathbb S(n)$ be a Seifert manifold  with at most $n$ exceptional fibers.
Then we obtain the following theorem.

	\begin{thm}\label{thm2}
	Let $M$ be a closed orientable 3-manifold. Then the following holds.
	\begin{enumerate}

\item[(1)]
$M$ has a type-$(0, 0, 0, 0)$ decomposition if and only if $M$ is homeomorphic to $\mathbb B$.
\item[(2)]
$M$ has a type-$(0, 0, 0, 1)$ decomposition if and only if $M$ is homeomorphic to $\mathbb B$ or $\mathbb B \# \mathbb L$.
\item[(3)]
$M$ has a type-$(0, 0, 1, 1)$ decomposition if and only if $M$ is homeomorphic to $\mathbb B$ or $\mathbb B \# \mathbb L$ or $\mathbb B \# \mathbb L_{1} \# \mathbb L_{2}$.
\item[(4)] 
$M$ has a type-$(0, 1, 1, 1)$ decomposition if and only if $M$ is homeomorphic to $\mathbb B$ or $\mathbb B \# \mathbb L$ or $\mathbb B \# \mathbb L_{1} \# \mathbb L_{2}$ or $\mathbb B \# \mathbb L_{1} \# \mathbb L_{2} \# \mathbb L_{3} $ or $\mathbb B \# \mathbb S(3)$.
\end{enumerate}
	\end{thm}
	
	Also we characterize  3-manifolds with  type-$(1, 1, 1, 1)$ decomposition.
\begin{thm}
Let $M$ be a closed orientable 3-manifold. 
Then $M$ has a type-$(1, 1, 1, 1)$ decomposition if and only if $M$ is homeomorphic to $\mathbb B$ or $\mathbb B \# \mathbb L$ or $\mathbb B \# \mathbb L_{1} \# \mathbb L_{2}$ or $\mathbb B \# \mathbb L_{1} \# \mathbb L_{2} \# \mathbb L_{3} $ or  $\mathbb B \# \mathbb L_{1} \# \mathbb L_{2} \# \mathbb L_{3} \# \mathbb L_{4}$ or $\mathbb B \# \mathbb S(4)$.
\end{thm}
We call the union of intersections of handlebodies of handlebody decomposition a {\it partition}.
If a partition is a simple polyhedron, any closed orientable 3-manifolds has type-$(0, 0, 0, 0)$ decomposition. Hence Theorem \ref{thm2} shows a difference between multibranched handlebody  decomposition and a handlebody decomposition with  simple polyhedron.

 This paper is organized as follows.
In Section 2, we introduce a multibranched surface and its moves. In Section 3, we describe the notion of multibranched handlebody decomposition and its stabilizations. After that, we show stableequivalence theorem in Section 4. After that, we  characterize a 3-manifolds by  multibranched handlebody decomposition with four handlebodies  in Section 5.

\section{multi-branched surface}
Let $\Bbb{R}^2_+$ be the closed upper half-plane $\{(x_1,x_2)\in \Bbb{R}^2 \mid x_2\ge 0\}$.
The  multibranched Euclidean plane, denoted by $\Bbb{R}^2_i$ $(i\ge 1)$, is the quotient space obtained from $i$ copies of $\Bbb{R}^2_+$ by identifying with their boundaries $\partial \Bbb{R}^2_+=\{(x_1,x_2)\in\Bbb{R}^2\mid x_2=0\}$ via the identity map.

\begin{Def}
A second countable Hausdorff space $X$ is called a multibranched surface if $X$ contains a disjoint union of simple closed curves $l_1,\ldots, l_n$ satisfying the following:
\begin{enumerate}
\item For each point $x\in l_1\cup \cdots \cup l_n$, there exists an open neighborhood $U$ of $x$ and a positive integer $i$ such that $U$ is homeomorphic to $\Bbb{R}^2_i$.
\item For each point $x\in X-(l_1\cup\cdots\cup l_n)$, there exists an open neighborhood $U$ of $x$ such that $U$ is homeomorphic to $\Bbb{R}^2$.
\end{enumerate}
We call $l_i$ a branched locus. The surfaces divided by branched loci are called regions.
\end{Def}

Sometimes multibranched surface is studied as a 2-stratifolds \cite{GG}.  A multibranched surface has been studied recently \cite{MO, O, GG}.  Ishihara, Koda, Ozawa and Shimokawa introduced the moves of multi-branched surface which does not change its regular neighborhood  \cite{KOKK}. We shall review the definition of XI and IX moves.

Let $X$ be a multibranched surface with brach loci $B=B_1\cup\cdots\cup B_m$ and regions $S=S_1\cup \cdots\cup S_n$, 
where $S$ is a (possibly disconnected or/and non-orientable) compact surface without disk components such that each component $S_j$ ($j=1,\ldots,n$) has a non-empty boundary.
Each point $x$ in $\partial S$ is identified with a point $f(x)$ in $B$ by a covering map $f:\partial S\to B$,
where $f|_{f^{-1}(B_i)}:f^{-1}(B_i)\to B_i$ is a $d_i$-fold covering ($d_i>2$). 
We call $d_i$ the {\it degree} of $B_i$.
We say that $B_i$ is {\it tribranched} or a {\it tribranch locus} if $d_i=3$. 
If all the branched loci in the multibranched surface are tribranched, we call it {\it tribranched surface}.
For each component $C$ of $\partial S$, the {\it wrapping number} of $C$ is $w_{C}$ if $f|_{C}$ is a $w_C$-fold covering for the branch locus $f(C)$.
Suppose $X$ is embedded in an orientable $3$-manifold $M$.
By \cite{MO}, then for each branch locus $B_i$ of $X$, the wrapping number of all components of $f^{-1}(B_i)$ is a divisor of $d_i$. 
We call the divisor $w_i$ the {\it wrapping number} of $B_i$. 
We say a branch locus $B_i$ is {\it normal} (resp. {\it pure}) if $w_i=1$ (resp. $d_i=w_i$).
In this paper, we assume that all the branched loci in a multibranched surface are normal.

\begin{Def}
Let $A$ be an annulus region of a multibranched surface $X$.
Suppose that each component of $\partial A$ is a tribranched locus. 
Then we can obtain another multibranched surface from $X$ by performing deformation retraction  of $A$ to the core circle of $A$ which sends $\partial A$ to a degree 4 branched locus.
We call this operation an {\it IX move} along $A$.

Let  $l$ be a branched locus of $X$ with degree 4, $S$ a  region  whose boundary contains $l$ and $A'$  a regular neighborhood of $l$ in $S$.
Then, we can consider the reverse operation of an IX move called {\it XI move} along $A'$. See Figure \ref{XI}.
\begin{figure}[httb]
					\centering
					\includegraphics[scale=0.5]{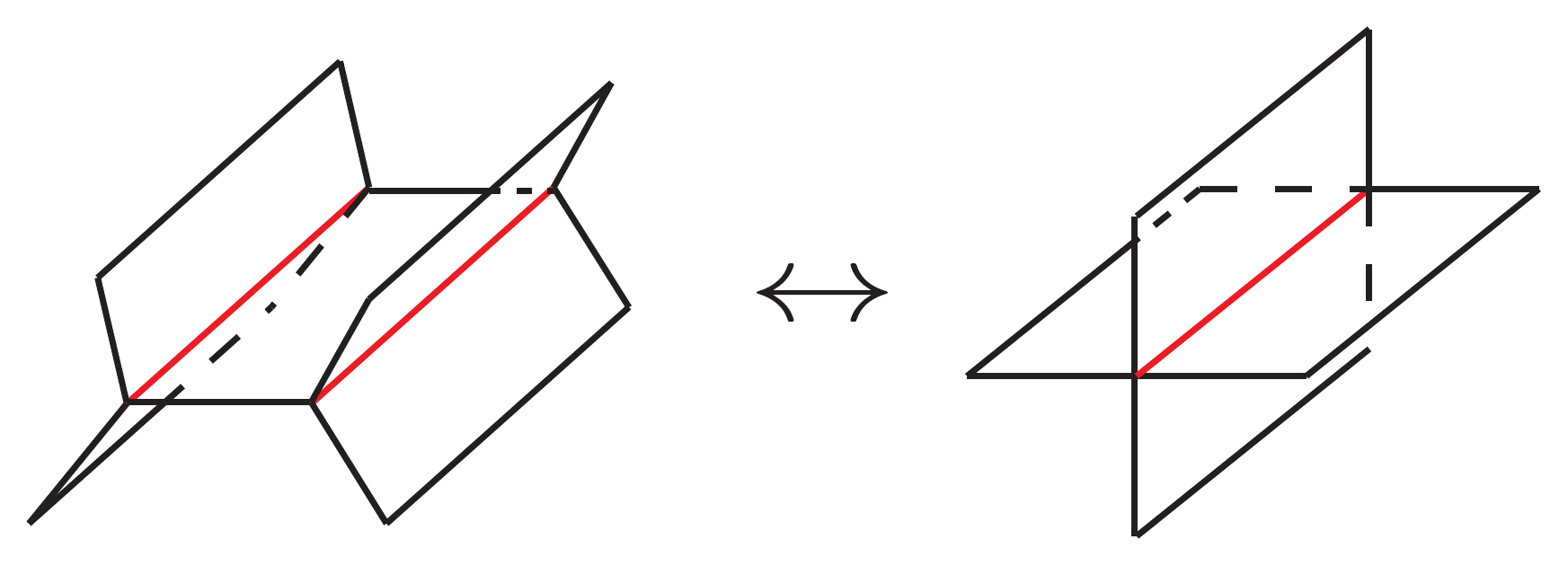}
					\caption
					{IX and XI move on a multibranched surface.}
					\label{XI}
			\end{figure}

\end{Def}
\begin{thm}[Theorem 1 in \cite{KOKK}]
Let $X, X'$ be multibranched surfaces in an orientable 3-manifold $M$, and let $N, N'$ be their regular neighborhoods respectively. If $N$ is isotopic to $N'$, then $X$ is transformed into $X'$ by a finite sequence of IX moves and XI moves and isotopies.
\end{thm}
\section{Handlebody decomposition whose partition is multi-branched surface}
Ishihara, Mishina, Koda, Ozawa, Sakata, Shimokawa and author introduced a handlebody decomposition of a 3-manifold whose partition is a simple polyhedron and showed the stable equivalent theorem for such decompositions \cite{Poly}. To show the stable equivalent theorem of handlebody decompositions, we use two  stabilizations and two moves of a simple polyhedron. In this section, we introduce the new decomposition of 3-manifolds called a multibranched handlebody decomposition.

\begin{Def}[Multibranched handlebody decomposition]
	Let $M$ be a closed orientable 3-manifold and $H_i$ a genus $g_i$ handlebody embedded in $M$ for $i=1, . . . , n$. $M=H_1\cup\cdots \cup H_n$ is a type-$(g_1, g_2, . . . , g_n)$ {\it multibranched handlebody decomposition} if the followings hold;
	\begin{enumerate}
		\item $H_i\cap H_j=\partial H_i\cap \partial H_j$ is a union of  possibly disconnected compact surfaces and simple closed curves. We denote $F_{ij}=H_i\cap H_j$.
		\item $H_{i1}\cap \cdots \cap H_{ik}$ is a union of simple closed curves in $M$ or emptyset for $k\geq 3$ and $\{i_1, . . . , i_k\}\subset \{1, . . . , n\}$. We call this simple closed curve branched locus.
	\end{enumerate}
	We call the union of $F_{ij}$ for all $i\neq j$ the {\it partition} of a multibranched handlebody decomposition. 
	We say that two multibranched handlebody decompositions of the same 3-manifold are isotopic to each other if each partition is isotopic to each other.
\end{Def}

\begin{Rem}
 It is clear that partition of a multibranched handlebody decomposition is a multibranched surface in $M$.
 Also, all the branched loci in a partition are normal since $H_i$ does not have self-intersection.
 
 We can obtain a type $(0, 0, g, g)$ multibranched handlebody decomposition from a genus $g$ Heegaard splitting. 
 Hence any closed orientable 3-manifold admits multibranched handlebody decomposition.
\end{Rem}

Let $M=H_1\cup\cdots \cup H_n$ be a multibranched handlebody decomposition and $P$  a partition of the multibranched handlebody decomposition.
Let $m$ be the maximal of degrees of branched loci. Then we say $M=H_1\cup\cdots \cup H_n$ is a {\it degree $m$ multibranched handlebody decomposition}. If $n=3$, the degree is also 3. In this paper, we consider the case where $n=4$.

In \cite{Poly}, we introduced type 0, 1 stabilizations for handlebody decomposition. We shall review the definition of stabilizations. 
\begin{Def}[stabilization]
\begin{enumerate}
\item The following operation is called a {\it type 0 stabilization} (Figure \ref{typeii}).
We take two points on the interior of $F_{ij}$ and connect them by a properly embedded boundary parallel arc $\alpha$ in $H_{i}$. Let $N(\alpha)$ be the regular neighborhood of $\alpha$ in $H_i$. we define a new handlebody decomposition $M = H_1'\cup\cdots \cup H'_{i}\cup \cdots  \cup H'_{j} \cup \cdots \cup H'_{n}$ by $H'_{i} := H_{i} \setminus int(N(\alpha))$, $H'_{j} := H_{j}\cup N(\alpha)$ and $H'_{k} := H_{k}$ for $k\neq i, j$. 
Then the n-tuple $(g_1, . . . , g_{i}, . . . ,  g_{j}, . . . ,  g_{n})$ is changed into $(g_1, . . . , g_{i}+1, . . . , g_{j}+1, . . . , g_{n})$ and the number of components of branched loci is not changed by this operation. 

\item The following operation is called a {\it type 1 stabilization} (Figure \ref{typeia}). 
We take two points on the branched loci and connect them by an arc $\alpha$ on $F_{jk}$. Let $N(\alpha)$ be the regular neighborhood of $\alpha$ in $M$. we define a new handlebody decomposition $M = H_1'\cup\cdots \cup H'_{i}\cup \cdots  \cup H'_{j} \cup \cdots \cup H'_{n}$ where $H'_{i} := H_{i} \cup N(\alpha)$, $H'_{j} := H_{j}\setminus int(N(\alpha))$ and $H'_{k} := H_{k}\setminus int(N(\alpha))$ $H_l'=H_l$ for $l\neq i, j, k$. 
Then the n-tuple $(g_1, . . . , g_{i}, . . . ,  g_{n})$ is changed into $(g_1, . . . , g_{i}+1, . . . , g_{n})$ and the number of components of branched loci is changed by $1$.
Conversely, if there exists a non-separating disk $D_{i} \subset H_{i}$ whose boundary intersects the set of branched loci exactly two points transversely, then $D_{i}$ can be canceled by an inverse operation of a type 1 stabilization. 
We call this operation a {\it type 1 destabilization}.
\end{enumerate}
\end{Def}
\begin{figure}[h]
\centering
\includegraphics[scale=0.5]{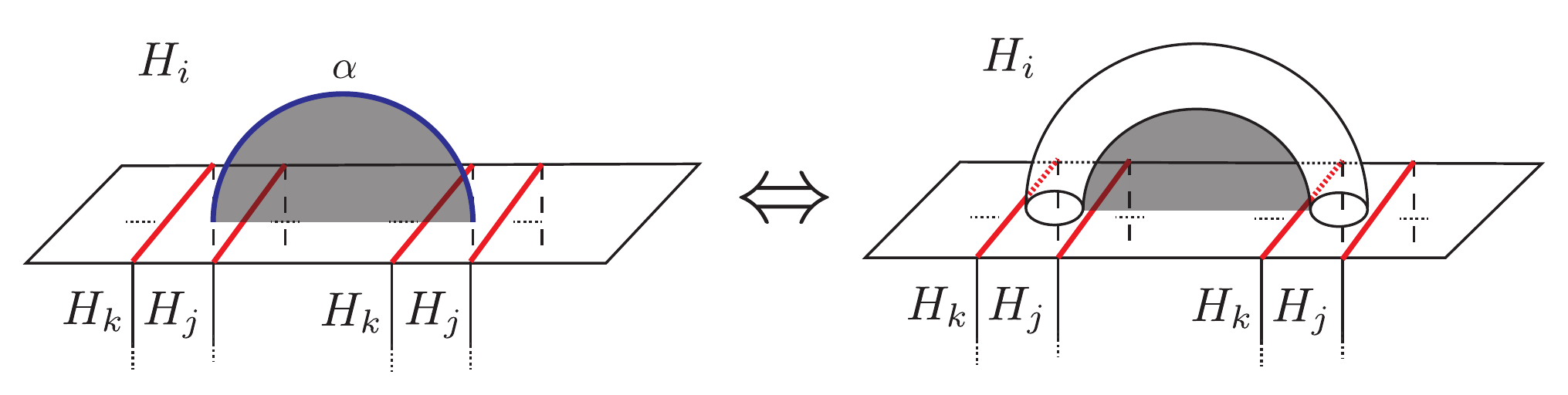}
\caption
{A type  0 stabilization along the arc $\alpha$.  A n-tuple $(g_1, . . . , g_{i}, . . . ,  g_{j}, . . . ,  g_{n})$ is changed into $(g_1, . . . , g_{i}+1, . . . , g_{j}+1, . . . , g_{n})$}
\label{typeii}
\end{figure}

\begin{figure}[h]
\centering
\includegraphics[scale=0.5]{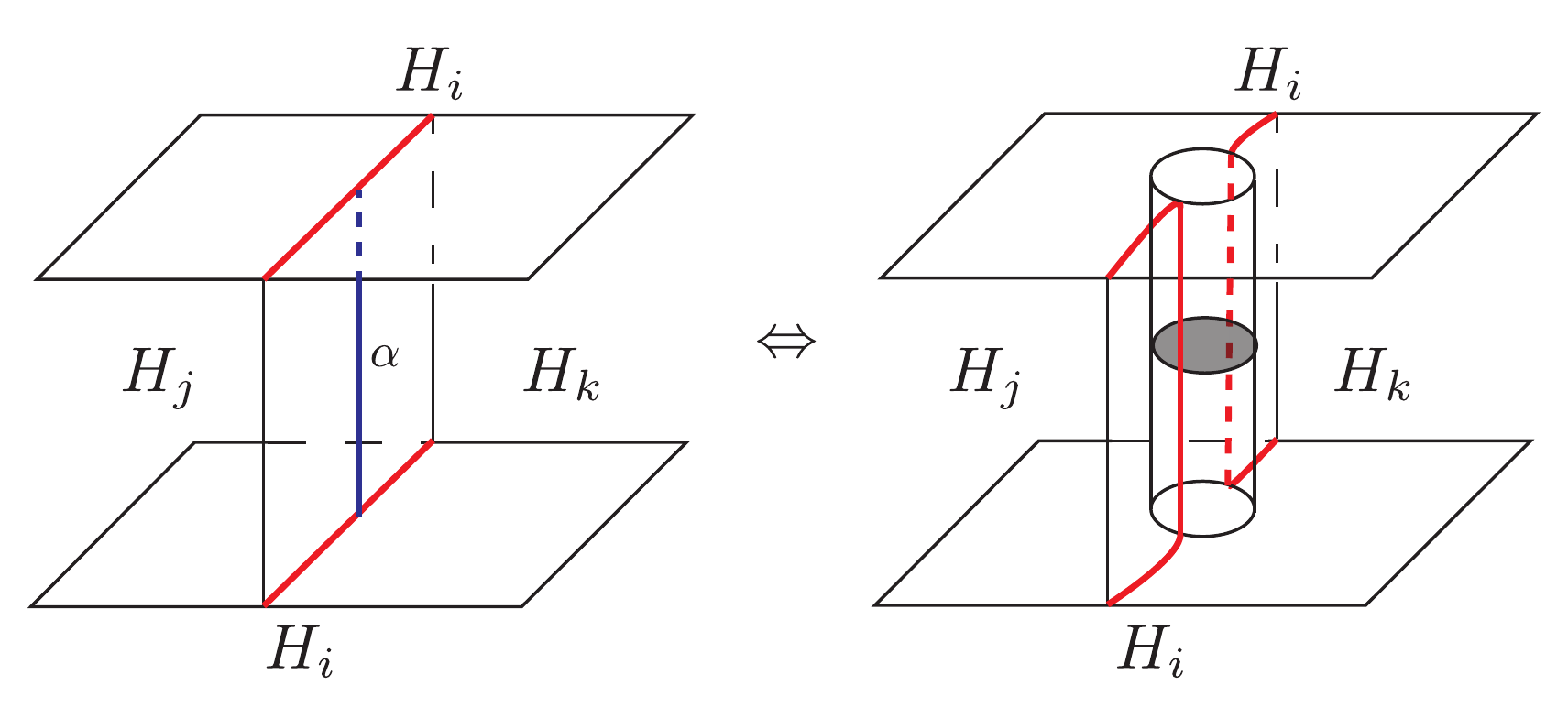}
\caption
{A type  1 stabilization along the arc $\alpha$. A n-tuple $(g_1, . . . , g_{i}, . . . ,  g_{n})$ is changed into $(g_1, . . . , g_{i}+1, . . . , g_{n})$.}
\label{typeia}
\end{figure}


\section{stable equivalent theorem}
In this section, we will prove Theorem  \ref{stbthm} by using stabilizations described above and XI, IX moves.
First, we consider the following lemma.
\begin{lem}\label{lem1}
	Let $X$ be a  partition of a multibranched handlebody decomposition and $A$ an annulus component of $F_{ij}$.
	Suppose that one of the components of $\partial A$ is a component of $H_i\cap H_j\cap H_l$ and the other is a component of $H_i\cap H_j \cap H_k$ for $l\neq k$.
	Also, suppose that $X$ is a tribranched surface.
	Then we can eliminate $A$ from $F_{ij}$ by XI  and IX moves so that the obtained partition $X'$ is also a tribranched surface. 
\end{lem}
\begin{Rem}
	After eliminating $A$ in Lemma \ref{lem1}, $F_{kl}$ shall have a new annulus component.
\end{Rem}
\begin{proof}[Proof of Lemma \ref{lem1}]
	After performing IX move along $A$, we can eliminate $A$ from $F_{ij}$ and obtain the branched locus $l$ with degree 4.
	By the assumption that X is a tribranched surface, $l=H_i\cap H_j\cap H_k\cap H_l$.
	Then there is a component of $F_{jl}$ whose boundary contains $l$.
	Then we can take a regular neighborhood  $A'$ of $l$ in the components of $F_{jl}$.
	An XI moves along $A'$ gives a tribranched surface X' as in conclusion.
	See Figure \ref{Fig_1}.
				\begin{figure}[h]
					\centering
					\includegraphics[scale=0.4]{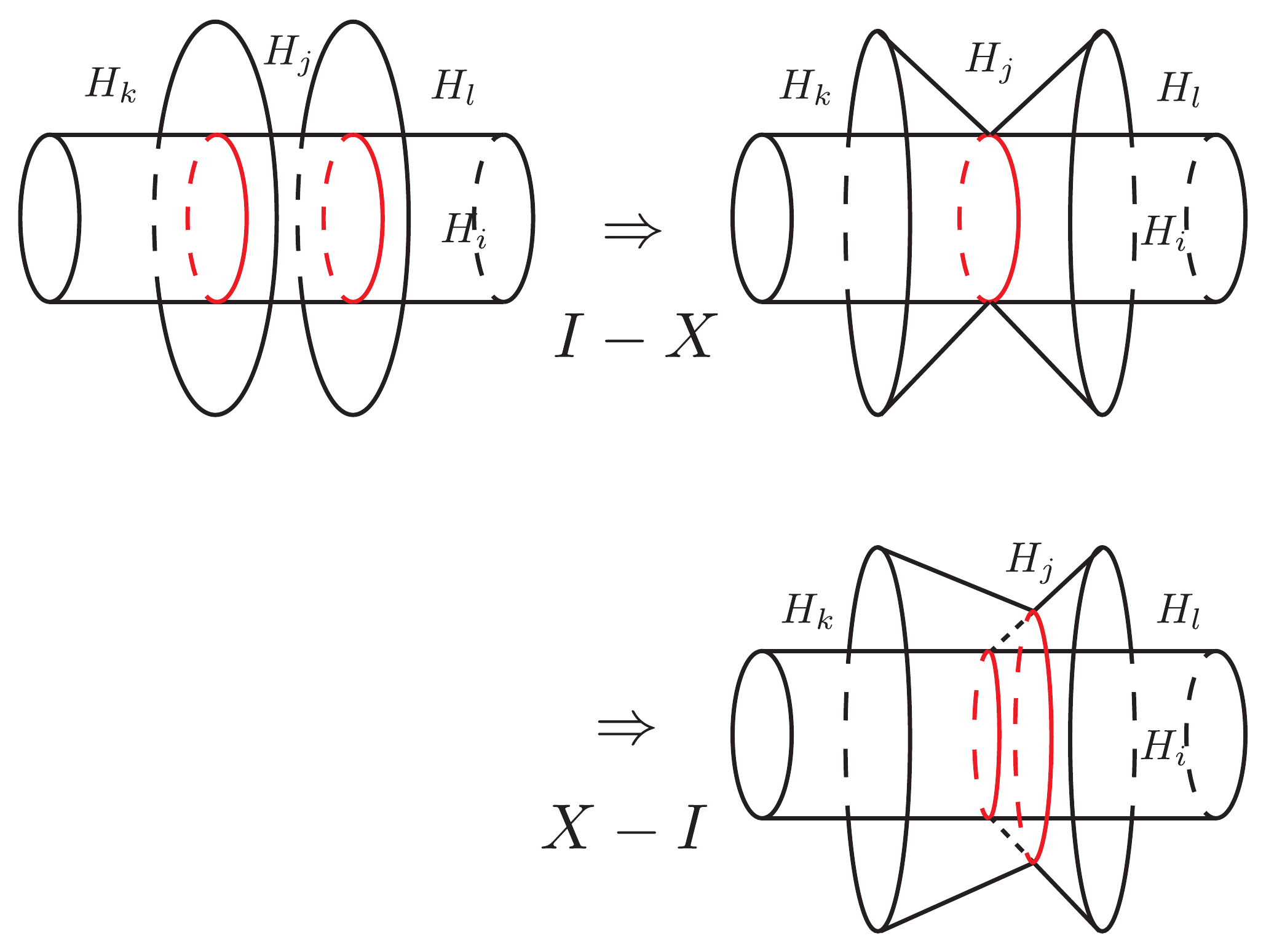}
					\caption
					{XI  and IX moves which deform $A$ into empty set.}
					\label{Fig_1}
			\end{figure}
\end{proof}

To prove the stable equivalence theorem, we shall define a local 1-handle attached to a handlebody.
\begin{Def}
	Let $M$ be a closed orientable 3-manifold and $H$ a handlebody embedded in $M$.
	We say a 1-handle $h$ attached to $\partial H$ is  local for $H$ if there exists a disk in the exterior of $H\cup h$ whose boundary intersects the boundary of a cocore of $h$ at one point.
	We call $D$ a dual disk of the pair $(H, h)$.
	\begin{figure}[httb]
					\centering
					\includegraphics[scale=0.5]{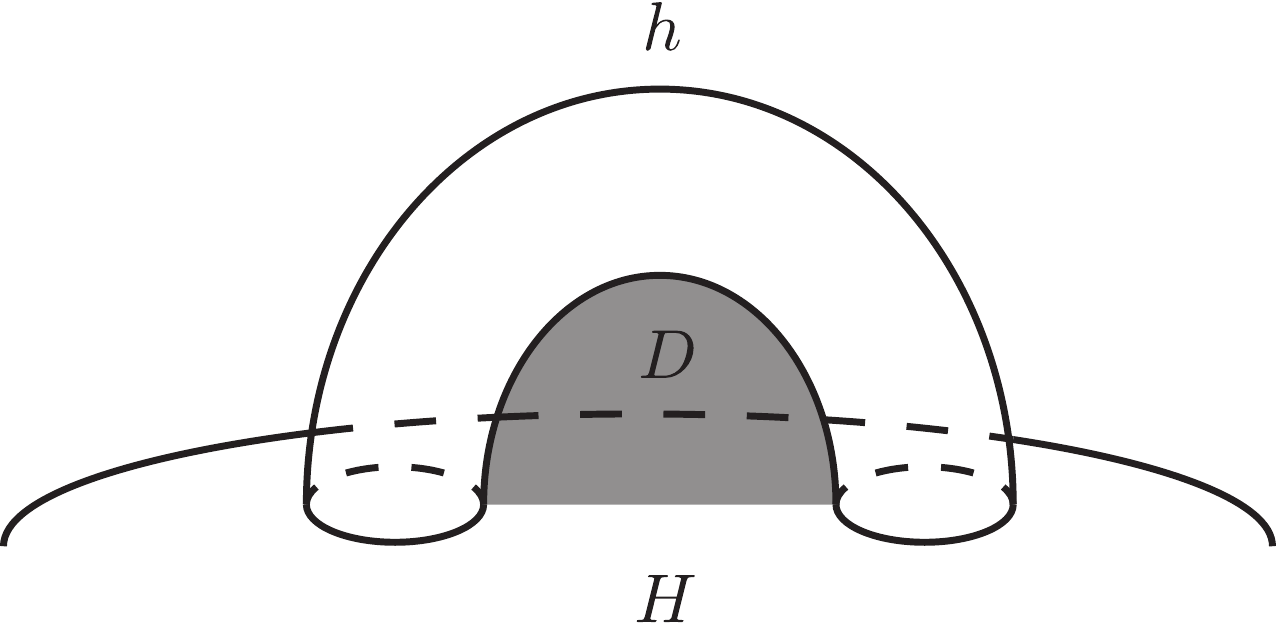}
					\caption{A 1-handle $h$ which is local for $H$}
					
			\end{figure}
\end{Def}
We note that two local 1-handles attached to the same handlebody is isotopic to each other after performing the handleslide on the handlebody.


The proof of Theorem \ref{stbthm} is divided into the following steps.
Let $M=H_1\cup H_2\cup H_3\cup H_4$ and $M=H_1'\cup H_2'\cup H_3'\cup H_4'$ be two multibranched handlebody decompositions.
\begin{enumerate}
	\item[Step 1:] We deform each of $F_{12}$, $F_{13}$, $F_{23}$, $F_{12}'$, $F_{13}'$ and $F_{23}'$ to a disk by performing type 1 stabilizations and XI, IX moves. Then $(H_1\cup H_2\cup H_3)\cup H_4$ and $(H_1'\cup H_2'\cup H_3')\cup H_4'$ become Heegaard splittings. Perform type 0 stabilizations until two Heegaard splittings $(H_1\cup H_2\cup H_3)\cup H_4$ and $(H_1'\cup H_2'\cup H_3')\cup H_4'$ are isotopic. Then we obtain $H_4=H_4'$.
	\item[Step 2:] Perform type 1 stabilizations and IX, XI moves until  each of $F_{24}$, $F_{34}$, $F_{24}'$ and $F_{34}'$ is a disk. Then each of $H_i$ and $H_i'$ is a handlebody attached to $H_4=H_4'$ at a disk for $i=2, 3$.
	\item[Step 3:] We show that each of 1-handles $h_i$ (resp. ${h_{i}}'$) of $H_i$ (resp. $H_i'$) is a local 1-handle attached to  $H_4\cup(H_i-h_i)$ (resp. $H_4'\cup(H_i'-h_i')$) for $i=2, 3$ so that dual disks of 1-handles are disjoint. Also we show that a 1-handle $h$ (resp. $h'$) whose cocore is $F_{23}$ (resp. $F_{23}'$) is a local 1-handle attached to $\partial (H_2\cup H_3\cup H_4)$ (resp. $\partial (H_2\cup H_3\cup H_4')$) so that a dual disk of $h$ (resp. $h'$) is a disjoint from dual disks of 1-handles of $H_i$ (resp. $H_i'$) for $i=2, 3$. This implies that $H_i=H_i'$ after performing handleslides for $i=2, 3$. 
	\item[Step 4:] In oder to perform handleslides $H_2$ on $H_4$, perform type 1 stabilization and XI, IX moves until $F_{14}$ is an annulus. After that, we perform type 0 stabilization until the genus of $H_2$ equals to that of $H_2'$. Then $H_2=H_2'$ after performing handleslides. 
	\item[Step 5:] In oder to perform handleslides $H_3$ on $H_2$, perform type 1 stabilizations and XI, IX moves until $F_{12}$ is a disk and $F_{14}$ is an empty set. After that, we perform type 0 stabilization until the genus of $H_3$ equals to that of $H_3'$.Then $H_3=H_3'$ after performing handleslides.
\end{enumerate}

\begin{proof}[Proof of Theorem \ref{stbthm}]
	Let $H_1\cup H_2\cup H_3\cup H_4$ and  $H_1'\cup H_2'\cup H_3'\cup H_4'$ be multibranched handlebody decompositions of the same 3-manifold $M$.
	Let $F_{ij}=H_i\cap H_j$ and $F_{ij}'=H_i'\cap H_j'$.
	After performing XI moves to each degree 4 branches, we can assume that all branched loci are tribranched.
	For each $\partial H_i$, we suppose $\partial H_i=F_{ij}\cup F_{ik}$ for $k\neq j$.
	Then we can assume that $H_1=F_{12}\cup F_{13}$.
 	Since  all the branched loci are tribranched, $\partial H_2=F_{12}\cup F_{23}$ and $\partial H_3=F_{13}\cup F_{23}$.
	This is a contradiction.
	Hence we can assume that $F_{1i}\neq\emptyset$ for $i=2, 3, 4$ without loss of generality as necessary after renaming indices.
	
	\underline{Step1}: We show the following claim to achieve step 1.
	\begin{claim}\label{cl1}
		We can assume $H_1\cup H_2\cup H_3$ is a handlebody after applying some XI and IX moves and type 1 stabilizations i.e. $(H_1\cup H_2\cup H_3)\cup H_4$ is a Heegaard splitting.
	\end{claim} 
		\begin{proof}
			We shall consider about $F_{12}$.
			Each component of $\partial F_{12}$ is a component of $\partial F_{13}$ or $\partial F_{14}$ also.
			Let $S_{12}$ be a component of $F_{12}$.
			Suppose that $\partial S_{12}\cap\partial F_{13}\neq\emptyset$ and $\partial S_{12}\cap\partial F_{14}\neq\emptyset$.
			Let $C$ be a component of $\partial S_{12}\cap\partial F_{13}$.
			Then we can take arcs properly embedded in $S_{12}$ which cuts open $S_{12}$ into a planar surface and their endpoints are in $C$.
			After performing type 1 stabilizations along such arcs, we can assume $S_{12}$ is a planar surface.
			Then we can take arcs properly embedded in $S_{12}$ which connects the components of $\partial S_{12}\cap\partial F_{13}$ (resp. $\partial S_{12}\cap\partial F_{14}$) and cut open $S_{12}$ into an annulus.
			After performing type 1 stabilizations along such arcs, we can assume $S_{12}$ is an annulus.
			Now one of the components of $\partial S_{12}$ is a component of $\partial F_{13}$ and the other is a component of $\partial F_{14}$.
			Then we can assume $S_{12}=\emptyset$ after performing XI  move and IX moves along $S_12$ by Lemma \ref{lem1}.
			If $S_{12}$ satisfies that  $\partial S_{12}\cap\partial F_{13}=\emptyset$ or $\partial S_{12}\cap\partial F_{14}=\emptyset$, $S_{12}$ is a disk after performing type 1 stabilizations.
			After performing the above procedure for all components of $F_{12}$,  $F_{12}$ is an empty set or a union of disks.
			
			Suppose that $F_{12}$ is an empty set.
			We shall consider about $F_{13}$. 
			We can take arcs properly embedded in $F_{13}$ which cut open each component of $F_{13}$ into a disk.
			Since $\partial H_1=F_{13}\cup F_{14}$, the endpoints of such arcs are contained in $\partial F_{14}$.
			Hence we can perform type 1 stabilization along such arcs.
			After performing type 1 stabilization along such arcs, we can assume that $F_{13}$ is a union of disks.
			Then we shall consider about $F_{23}$.
			Since $\partial H_2=F_{23}\cup F_{24}$, we can take arcs properly embedded in $F_{23}$ so that such arcs cut open $F_{23}$ into a disks.
			Hence $F_{23}$ can be deformed into an empty set or a union of disks without changing $F_{12}$ and $F_{13}$ in the same way as before.
			Since $F_{12}=\emptyset$, $F_{13}$ is a union of disks, $F_{23}$ is an empty set or a union of disks, $H_1\cup H_2\cup H_3$ is a handlebody.
			
			Next, we suppose that $F_{12}$ consists of disks.
			Let $S_{13}$ be a component of $F_{13}$.
			Suppose $\partial S_{13}\cap\partial F_{14}=\emptyset$.
			Then we can assume that each of the components of $\partial S_{13}$ is also a component of $\partial F_{12}$.
			Since $F_{12}$ is a union of disks, $\partial H_1=S_{13}\cup F_{12}$.
			This contradicts that $F_{14}\neq\emptyset$.
			Hence we can assume that $\partial S_{13}\cap\partial F_{14}\neq\emptyset$.
			Suppose that $\partial S_{13}\cap\partial F_{12}\neq\emptyset$ and $\partial S_{13}\cap\partial F_{14}\neq\emptyset$.
			Let $C'$ be a component of $\partial S_{13}\cap\partial F_{14}$.
			Then we can take arcs properly embedded in $S_{13}$ which cuts open $S_{13}$ into a planar surface and satisfies their endpoints are in $C'$.
			After performing type 1 stabilizations along such arcs, we can assume $S_{13}$ is a planar surface without changing $F_{12}$.
			Then we can take arcs properly embedded in $S_{13}$ which connects the components of $\partial S_{13}\cap\partial F_{12}$ (resp. $\partial S_{13}\cap\partial F_{14}$) each other and cut open $S_{13}$ into an annulus with keeping $F_{12}$ as a union of disks.
			After perfroming type 1 stabilizations along such arcs, we can assume $S_{13}$ is an annulus.
			Then we can deform $S_{13}$ into an empty set by performing XI move and IX moves along $S_{13}$ without changing $F_{12}$ by Lemma \ref{lem1}.
			Suppose that $\partial S_{13}\cap\partial F_{12}=\emptyset$ and $\partial S_{13}\cap\partial F_{14}\neq\emptyset$.
			Then we can take arcs properly embedded in $S_{13}$ which cuts open $S_{13}$ into a disk so that endpoints  of such arcs are contained in $\partial F_{14}$.
			After performing  type 1 stabilizations along such arcs, $S_{13}$ becomes a disk.
			Hence $F_{13}$ is an empty set or a union of disks.
			Then we shall consider $F_{23}$.
			We note that $F_{24}$ or $F_{34}$ is not an empty set.
			Hence we can assume that $F_{24}\neq\emptyset$.
			Let $S_{23}$ be a component of $F_{23}$.
			If $S_{23}\cap F_{24}=\emptyset$, $\partial H_2=S_{23}\cup F_{12}$ since $F_{12}$ is a union of disks.
			In particular, $F_{24}=\emptyset$. This is a contradiction.
			Hence $S_{23}\cap F_{24}\neq\emptyset$.
			Then we can take arcs properly embedded in $S_{23}$ which cuts open $S_{23}$ into an annulus or disk so that endpoints  of such arcs are contained in $\partial F_{24}$.
			We can deform $F_{23}$ into an empty set or a union of disks by performing type 1 stabilizations along such arcs and XI  move and IX moves.
			Now, $F_{ij}$ is an empty set or a union of disks for $\{i, j\}\subset \{1, 2, 3\}$.
			After performing type 1 stabilization, we can assume $F_{ij}$ is an empty set or exactly one disk for $\{i, j\}\subset \{1, 2, 3\}$.
			Hence we can assume that $H_1\cup H_2\cup H_3$ is a handlebody.
		\end{proof}
			\begin{figure}[httb]
					\centering
					\includegraphics[scale=0.5]{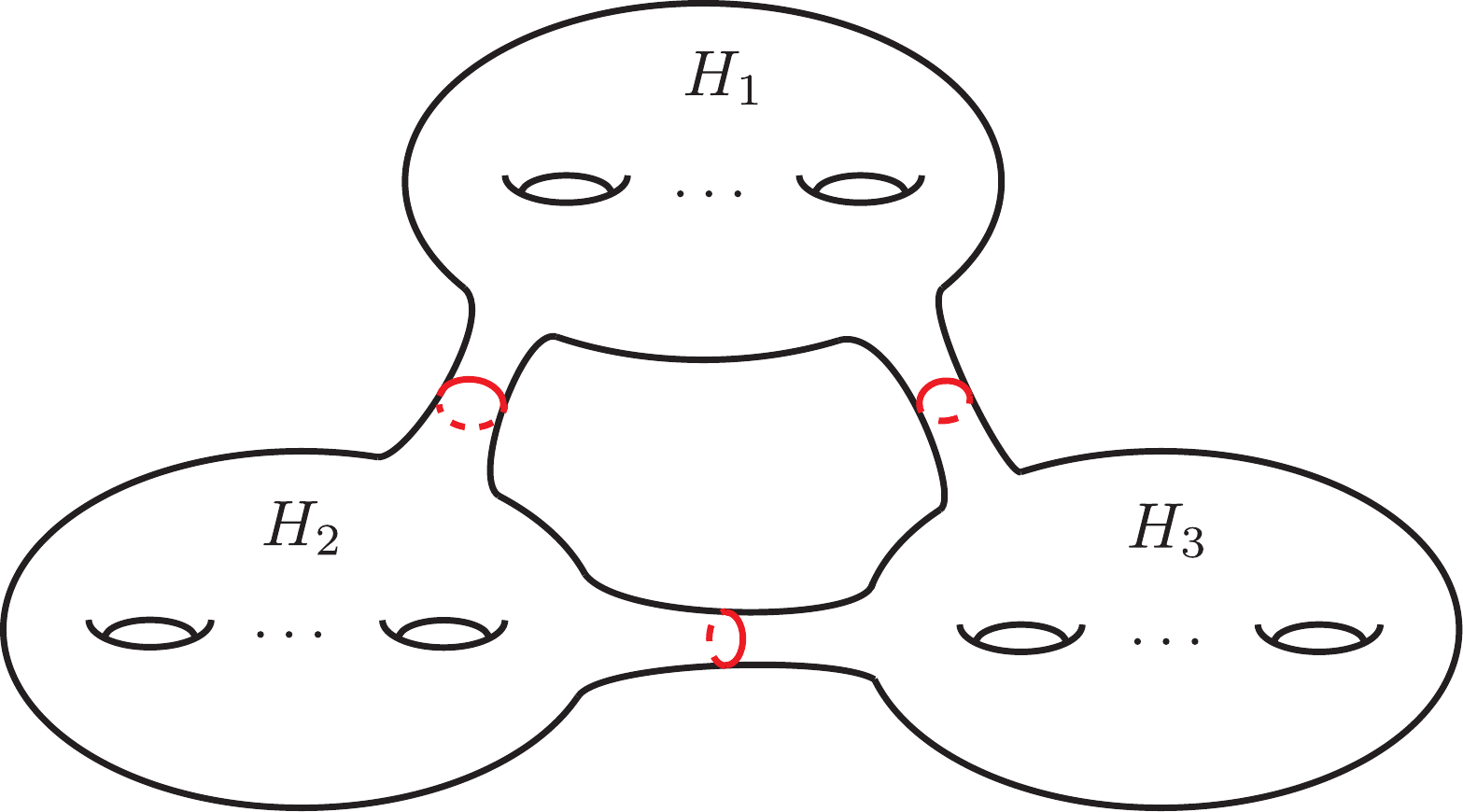}
					\caption
					{$H_1\cup H_2\cup H_3$ is a handlebody.}
					\label{Fig_2}
			\end{figure}
	Similarly, we can assume $H_1'\cup H_2'\cup H_3'$ is a handlebody.
	By Claim \ref{cl1} and a stable equivalence of Heegaard splittings, we can assume $H_4=H_4'$ after performing type 0 stabilizations finitely many times.
	From now, we assume that each of $F_{12}$, $F_{13}$ and $F_{23}$ is a disk.
	
	\underline{Step 2}: We show the following claim to achieve step 2.
	\begin{claim}\label{cl2}
		We can deform $F_{24}$ and $F_{34}$ into disks by type 1 stabilizations and XI  and IX moves.
	\end{claim}
		\begin{proof}
			Now, each of $F_{12}$, $F_{13}$ and $F_{23}$ is a disk.
			Then we can take arcs properly embedded in $F_{24}$ which cut open $F_{24}$ into a disk and annulus and satisfies their endpoints are contained in $\partial F_{12}$.
			After applying type 1 stabilizations along such arcs, we can assume $F_{24}$ is a union of a disk and an annulus.
			One of the boundaries of the annulus component of $F_{24}$ is the component of $F_{12}$ and the other is the component of $F_{34}$.
			By Lemma \ref{lem1}, after applying XI  and IX moves along the annulus, we can assume $F_{24}$ is a disk and $F_{13}$ has the new annulus component.
			
			Now,  $F_{13}$ is the union of a disk and an annulus.
			We note that each of the components of $\partial F_{34}$ is a component of $\partial F_{13}$.
			Then we can take arcs properly embedded in $F_{34}$ which cut open $F_{34}$ into a disk.
			After applying type 1 stabilizations along such arcs, we can assume that $F_{34}$ is a disk. 
			(See Figure \ref{Fig_2})
			\begin{figure}[httb]
					\centering
					\includegraphics[scale=0.5]{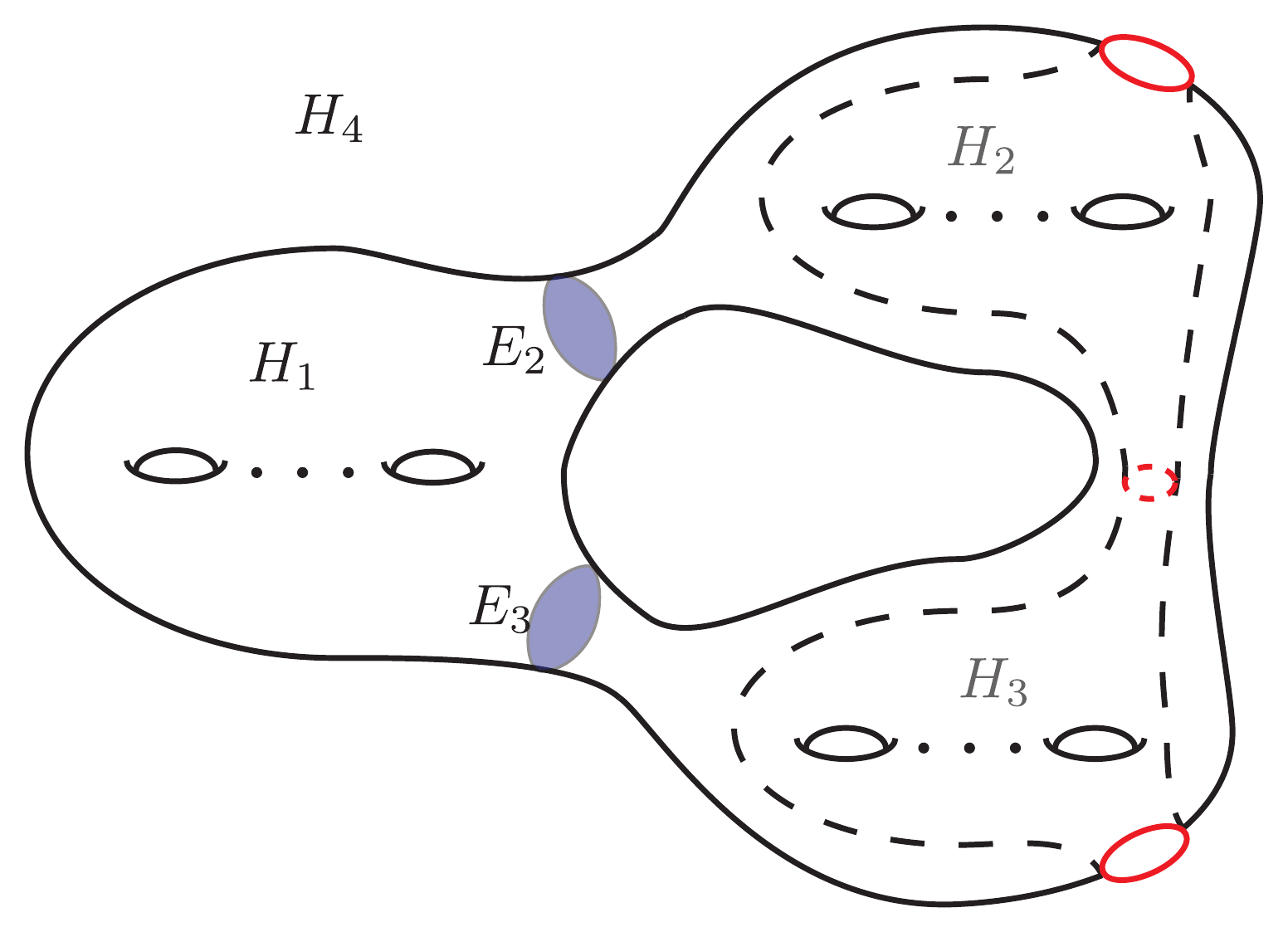}
					\caption
					{The situation of handlebody decomposition after Claim \ref{cl2}. The red curves in this figure are branched loci.}
					\label{Fig_2}
			\end{figure}
		\end{proof}
			
		\underline{Step 3}: We show the following two claims to achieve step 3.
		\begin{claim}\label{cl3}
		Let $D_{i1}, . . . , D_{ig_i}$ be a  complete meridian disks system of $H_i$ so that $\partial D_{ij}\subset F_{1i}$ for $i\in\{2, 3\}$ and $j\in\{1, . . . , g_i\}$ and $D$ the union $\cup_{i, j} D_{ij}$ for all $D_{ij}$.
		Then there exist disjoint meridian disks $E_{ij}$ $( i\in\{2, 3\}, j \in\{1, . . . , g_i\} )$ of $H_1$ such that $\partial E_{ij}\subset F_{14}\cup F_{1i}$ and $E_{ij}\cap D=E_{ij}\cap D_{ij}$ is one point.
	\end{claim}
	\begin{proof}
		There exist  mutually disjoint disks $E_2$ and $E_3$ in $H_1$ such that  $E_2\cup E_3$ cuts off a handlebody $W$
		from $H_1$ so that $(W, F_{12}\cup F_{13})$ is homeomorphic to $((F_{12}\cup F_{13})\times [0, 1], (F_{12}\cup F_{13})\times \{0\})$ (See Figure \ref{Fig_2}).
		Then we can take mutually disjoint non-separating arcs $\alpha_{i1}, . . . , \alpha_{ig_i}$ properly embedded in $F_{1i}$ so that 
		$\alpha_{ij}\cap D=\alpha_{ij}\cap D_{ij}$ is exactly one point and $\partial \alpha_j\subset F_{i4}$ for $i=2, 3$. Let $E_{ij}$ be a disk corresponding to $\alpha_{ij}\times [0, 1]$ so that $E_{ij}\cap E_i=\emptyset $ for each $i\in\{2, 3\}$. 
		Then the statement holds since $\partial W-(E_1\cup E_2)\subset F_{12}\cup F_{13}\cup F_{14}$.
		We note that $E_{ij}$ is a meridian disk of $H_i$ since each of $\alpha_{ij}$ is a non-separating arc. 
	\end{proof}
	\begin{claim}\label{cl4}
		Let $D_{i1}, . . . , D_{ig_i}$ be a  complete meridian disks system of $H_i$ so that $\partial D_{ij}\subset F_{1i}$ for $i\in\{2, 3\}$ and $j\in\{1, . . . , g_i\}$ and $D$ the union $\cup_{i, j} D_{ij}$ for all $D_{ij}$. 
		Then there is a meridian disk $D'$ of $H_1$ which satisfies the following. 
		\begin{enumerate}
			\item $D'\cap D=\emptyset$.
			\item $\partial D\subset F_{12} \cup F_{13}\cup F_{14}$.
			\item $D\cap F_{34}$ is exactly one point.
			\item $D'$ does not intersects $E_{ij}$ constructed in Claim \ref{cl3}
		\end{enumerate}
	\end{claim}
	\begin{proof}
		Let $W$ be the handlebody in a proof of Claim \ref{cl3} and each of $E_{ij}$ is a disk obtained by Claim \ref{cl3}.
		Since $E_{ij}\cap H_2$ and $E_{ij}\cap H_3$ does not intersects $F_{23}$, there is a non-separating arc properly embedded in $F_{12}\cup F_{13}$ which does not intersects all $E_{ij}$ in Claim \ref{cl3}.
		Hence we can take a non-separating arc $\beta$ properly embedded in $ F_{12} \cup F_{13}$ so that $\beta\cap D=\emptyset$, $\beta\cap F_{23}$ is one point and one of the endpoints of $\partial \beta$ is in $F_{24}$ and the other is in $F_{34}$.
		Let $D'$ be a disk corresponding to $\beta\times [0, 1]$ so that $D'\cap E_i$ for each $i=2, 3$.
		Then the statement holds.
	\end{proof}
	
	We call a regular neighborhood of each of $D_{ij}$ in $H_i$ a 1-handle of $H_i$ for $i=2, 3$.
	We call a regular neighborhood of $F_{23}$ in $H_2\cup H_3$ a 1-handle connecting  $H_2$ and $H_3$.
	Claim \ref{cl3}, \ref{cl4} implies that any 1-handle of  $H_2$ and $H_3$ is a local 1-handle for $H_4$.
	The disks $E_{ij}$'s and $D'$ in Claim \ref{cl3}, \ref{cl4} are dual disks for 1-handles of $H_i$ for $i=2, 3$.
	Similarly, we can also take such dual disks for the 1-handles of $H_2'$ and $H_3'$ and a 1-handle connecting $H_2'$ and $H_3'$.
	Let $S_1$ be the surface $F_{14}$ at this stage.
	
	\underline{Step 4}: We shall show $H_2=H_2'$ in this step.
	Since $F_{24}$, $F_{34}$, $F_{24}'$ and $F_{34}'$ are disks, we can assume that $F_{24}=F_{24}'$ and $F_{34}=F_{34}'$.
 	We can take arcs properly embedded in $S_1$ so that the arcs cut open $S_1$ into an annulus and their endpoints lie in $\partial F_{24}=\partial F_{24}'$.
	We perform type 1 stabilizations for $H_2$ and $H_2'$ along such arcs. 
	Then $F_{14}=S_1-F_{12}(=S_1-F_{12}')$ becomes an annulus $A$ such that one of the boundaries of $A$ is a component of $\partial F_{13}$ (resp. $F_{13}$') and the other is a component of $\partial F_{23}$ (resp. $F_{14}$'). See Figure \ref{Fig_6}.
	
	According to the stabilizations of $H_2$ (resp. $H_2'$), there exists a separating disk $D_2$ and $D_2'$ in $H_2$ and $H_2'$ respectively which cut off  handlebodies $V_2$ and $V_2'$ from $H_2$ and $H_2'$ respectively so that $(V_2, S_1-A)$ and $(V_2', S_1-A)$ are homeomorphic to $((S_1-A)\times [0, 1], (S_1-A)\times \{0\})$.
	Since $H_4=H_4'$, $V_2=V_2'$. See Figure \ref{Fig_6}.
	
	We note that dual disks of 1-handles of $H_2-V_2$ and $H_2'-V_2'$ induce dual disks of $H_2-V_2$ and $H_2'- V_2'$ in $H_1$ for $V_2=V_2'$ respectively.
	Hence, 1-handles of $H_2-V_2$ and $H_2'-V_2'$ are local for $V_2=V_2'$. 
	If necessary, we perform type 0 stabilizations of $H_2$ or $H_2'$ until genus of $H_2$ and $H_2'$ are the same. 
	The 1-handles of $H_2$ and $H_2'$ which are obtained by type 0 stabilizations are local for $V_2$ and $V_2'$ respectively by a definition of a type 0 stabilization.
	After that we perform handle sliding 1-handles of $H_2-V_2$ on $V_2=V_2'$ until $H_2-V_2=H_2'-V_2'$.
	Since $V_2=V_2'$, $H_2=H_2'$.
	\begin{figure}[h]
					\centering
					\includegraphics[scale=0.5]{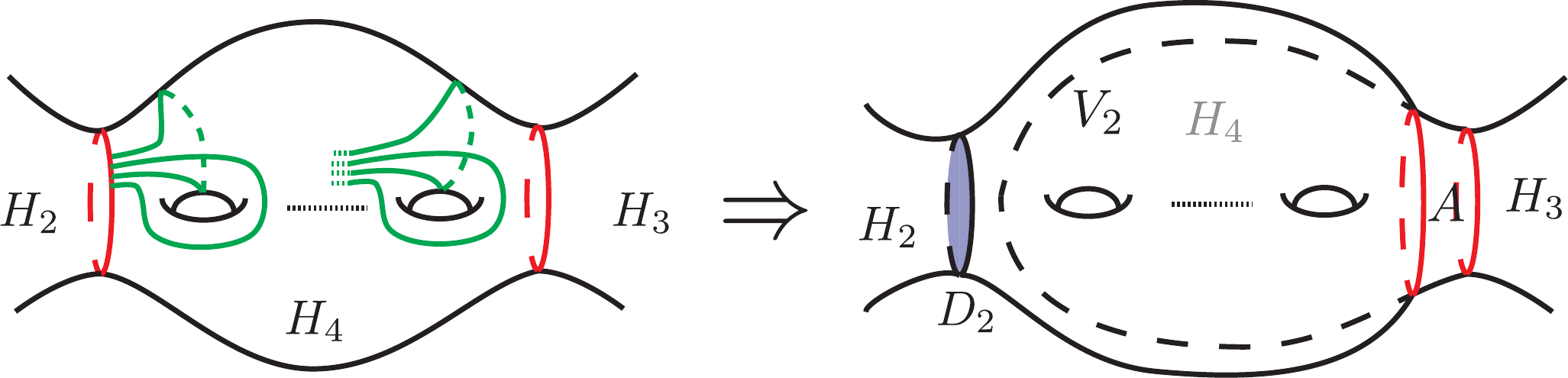}
					\caption
					{We perform type 1 stabilizations along green arcs in this left figure. After that $H_2$ is divided two handlebodies by $D_2$ in the right figure.}
					\label{Fig_6}
			\end{figure}
	
	Let $S_2$ be the surface $F_{12}$ at this stage.
	
	\underline{Step 5}: We shall show $H_3=H_3'$ in this step.
	After performing XI  and IX move along $A$, we can eliminate $A$ from $F_{14}$ by Lemma \ref{lem1}.
	After that, $F_{14}$ becomes an emptyset.
	Then we can take arcs properly embedded in $F_{12}$ (resp. $F_{12}'$) which cut open $F_{12}$ (resp. $F_{12}'$) into a disk $D^2$ and their endpoints lie in $\partial H_3$ (resp. $\partial H_3'$).
	We can perform type 1 stabilizations along such arcs.
	
	According to the stabilizations of $H_3$ and $H_3'$, there exists a separating disk $D_3$ and $D_3'$ in $H_3$ and $H_3'$ respectively which cut off handlebodies $V_3$ and $V_3'$ from $H_3$ and $H_3'$ respectively so that $(V_3, S_2-D^2)$ and $(V_3', S_2-D^2)$ are homeomorphic to $((S_2-D^2) \times [0, 1], (S_2-D^2) \times \{0\})$.
	Since $H_2=H_2'$ and $H_4=H_4'$, $V_3=V_3'$.

	By Claim \ref{cl3}, 1-handles of $H_3-V_3$ and $H_3'-V_3'$ are local for $V_3=V_3'$.
	Also, by Claim \ref{cl4}, 1-handles connecting $H_3$ and $H_2$ (resp. $H_2'$ and $H_3'$) is local for $V_3$ (resp. $V_3'$).
	If necessary, we perform type 0 stabilization of $H_3$ or $H_3'$ until genus of $H_3$ and $H_3'$ are the same. 
	After that, we perform handle sliding 1-handles of $H_3-V_3$ on $V_3=V_3'$ until $H_3-V_3=H_3'-V_3'$.
	Since $V_3=V_3'$, $H_3=H_3'$.

	Finally we have $H_1=H_1'$ automatically from $H_1=H_1'$, $H_2=H_2'$ and $H_4=H_4'$.
	This implies that partitions of two multibranched handlebody decompositions are isotopic to each other.
\end{proof}
\section{characterization of 3-manifolds with multibranched handlebody decompositions with four handlebodies}
	In this section,  we characterize 3-manifolds with multibranched handlebody decompositions with four handlebodies. 
	We consider the case where the genera of handlebodies are at most one. 
	First, we consider the decompositions such that one of the handlebodies is a 3-ball.
	\begin{prop}\label{prop1}
		Let $M$ be a closed, connected, orientable 3-manifold.
		If $M$ has a type-$(0, g_2, g_3, g_4)$ multibranched handlebody decomposition, then $M$ has a type-$(g_2, g_3, g_4)$ decomposition.
	\end{prop}
	\begin{proof}
		After performing XI moves, we can assume that all branched loci are tribranched.
		$H_1$ is a 3-ball.
		Suppose that $F_{14}=\emptyset$.
		Then $\partial H_4=F_{24}\cup F_{34}$.
		If one of $F_{24}$ and $F_{34}$ is an emptyset, $M$ is not connected.
		This is a contradiction.
		Then each of $F_{24}$ and $F_{34}$ is not an emptyset.
		If any components of $F_{23}$ do not intersect component of $F_{24}$, $\partial F_{24}$ is contained $\partial F_{12}$. 
		This contradicts that $F_{14}=\emptyset$.
		Hence any components of $F$ of $F_{23}$ have a boundary component which is also a boundary component of $F_{24}$.
		Then we can take an arc properly embedded in the component of $F_{23}$ which connects $H_1$ and $H_4$.
		Let $N$ be a regular neighborhood of such arc.
		Then, let $H_2'=H_2 - N$, $H_3'=H_3 - N$ and $H_4'=H_4\cup N\cup H_1$.
		Since $H_4\cap N$ and $N\cap H_1$ is a disk and $H_4\cap H_1=\emptyset$, $H_i'\cong H_i$ for $i=2, 3, 4$.
		Hence $M$ has a type-$(g_2, g_3, g_4)$ decomposition.
		
		Then, we can suppose that $F_{1i}\neq\emptyset$ for $i=2, 3, 4$.
	 	We take a regular neighborhood  $N(F_{12})$  of $F_{12}$ in $\partial H_1$ so that branched loci contained in $N(F_{12})$ is only $\partial F_{12}$.
		Hence $N(F_{12})-F_{12}$ is a union of annuli which are the regular neighborhood of $\partial F_{12}$ in $F_{13}$ or $F_{14}$. 
		Let $A_1, . . . , A_k$ be such annuli.
		There are mutually disjoint disks properly embedded in $H_1$ whose boundaries are components of $\partial N(F_{12})$.
		Such disks cut open $H_1$ into some 3-balls.
		We call such 3-balls $C_1, . . . , C_n$ if their boundary contain the component of $F_{12}$.
		Otherwise, we call $C_1', . . . , C_m'$.
		
		There exist properly embedded essential arcs in each components of $N(F_{12})-F_{12}=A_1\cup\cdots\cup A_k$.
		We take a subset of such arcs $\{\alpha_1, . . . , \alpha_m\}$ so that one of the endpoints of $\alpha_i$ are contained in  $C_i'$.
		Let $N$ be a regular neighborhood of $\alpha_1\cup \cdots\cup \alpha_m$.
		Let $H_2'=H_2\cup N\cup C_1'\cup \cdots \cup C_m'$, $H_3'=H_3-N$, $H_4'=H_4-N$.
		Since $H_2\cap N$ are disks and $N\cap C_i'$ is a disk, $H_i'\cong H_i$ for $i=2, 3, 4$.
		
		Next, we shall consider the 3-balls $C_1, . . . , C_n$.
		Suppose that at least one of the 3-ball $C_i$ does not contain the subsurface of $F_{13}$.
		Let $C_1$ be a such 3-ball.
		If there is a component of $F_{14}$ which  has an intersection with $C_1$  intersects $F_{13}$, 
		we can take an arc properly embedded in $F_{14}-(F_{14}\cap C_1)$ which connects $C_1$ and $H_3$.
		Let $N'$ be a regular neighborhood of a such arc.
		Then we can replace $H_3'=H_3\cap N \cap C_1$.
		Next, Suppose that any components of $F_{14}$ which has an intersection with $C_1$ do not  intersect $F_{13}$.
		Then the components of $F_{14}$ in $\partial C_1$ adjacent to a component of $F_{12}$.
		Let $F'$ be the component of $F_{14}$ in $\partial C_1$.
		Let $C_2$ be a 3-ball which contains a component of $F_{12}$ adjacent to $F'$.
		Suppose $\partial C_2$ contains a subsurface of $F_{13}$.
		We can take a regular neighborhood $N(F')$ of $F'$ in $\partial C_2$.
		Then there is a disk which cuts open $C_2$ into two  3-balls. Let  $C_2^1$ be one of the 3-ball which contains subsurface of $F'$ and  $C_2^2$ the other.
		We can take an arc which connect $H_3$ and $C_2^1$, and an arc properly embedded in $F'-(F'\cap C_2)$ which connects $C_2^1$ and $C_1$.
		Let $N'$ be a regular neighborhood of such arcs.
		Then we can replace $H_3'=H_3\cap N'\cap (C_1\cap C_2^1)$.
		Even if $C_2$ does not contains a subsurface of $F_{13}$, we can proceed the steps above. 
		Hence  we can assume any $C_i$ contains $F_{13}$ and $F_{14}$ for $i=1, . . . , n$.

		The intersection $C_i\cap H_2'$ equals to $(C_i\cap N)\cup (C_i\cap (C_1'\cup\cdots\cup C_m'))\cup S_i$ where $S_i$ is a component of $F_{12}$ in $C_i$.
		Then $C_i\cap (H_3\cup H_4)$ equals to $C_i\cap((A_1\cup\cdots \cup A_k)-N)$.
		If $\alpha_i$ is contained in $A_k$,  $N(\alpha_i)$ cuts open $A_k$ into a disk.
		This implies that some of the intersections $C_i\cap (H_3\cup H_4)$ are disks and the others are annuli.
		Let $F^i_j$ be intersections $C_i$ and $H_j'$ for $j=2, 3, 4$.
		Then $\partial C_i$ can be seen as in Figure \ref{Fig_3}.
		\begin{figure}[h]
					\centering
					\includegraphics[scale=0.9]{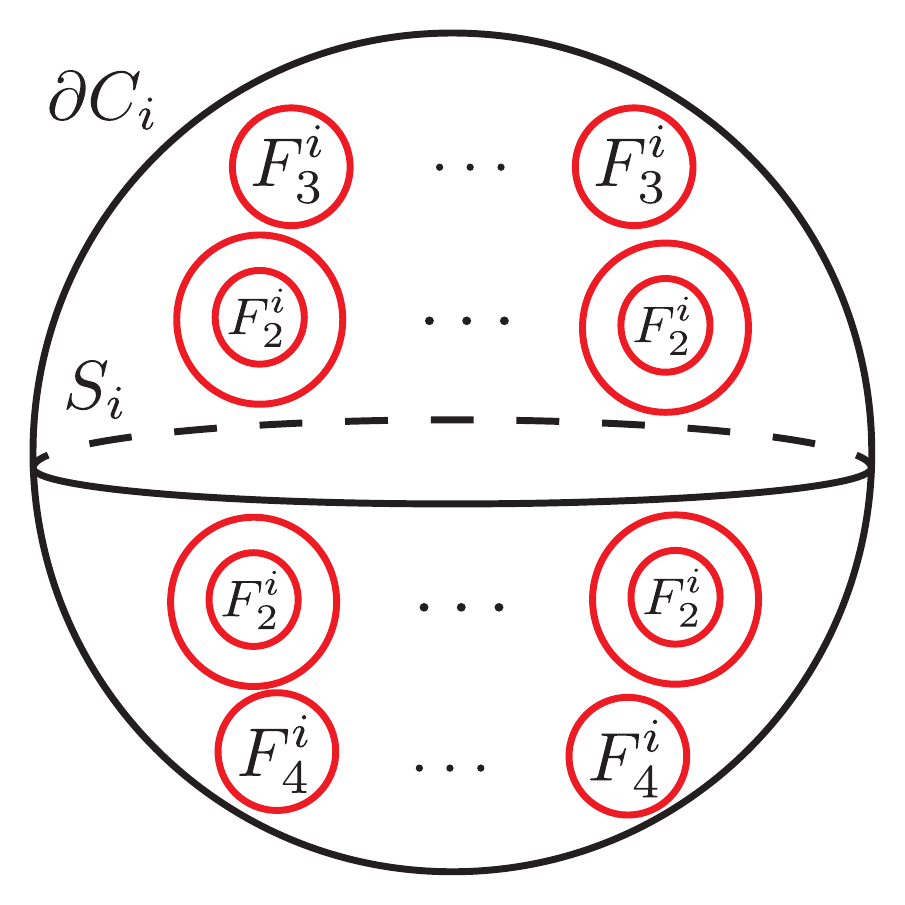}
					\caption
					{$\partial C_i$ intersects  $H_3'$ and $H_4'$ in annuli and disks.}
					\label{Fig_3}
		\end{figure}
		
		We take  a regular neighborhood $N(F^i_4)$ of  $F^i_4$ so that branched loci in $N(F^i_4)$ is only $\partial F^i_4$ for $i=1, . . . , n$.
		We can take mutually disjoint disks properly embedded in $C_i$ whose boundary equals to the components of $\partial N(F^i_4)$ which is contained in $F_{12}$ for $i=1, . . . , n$.
		Such disks cuts open $(C_i-N)$ into 3-balls for $i=1, . . . , n$.
		
		Let $B_1, . . . , B_l$ be such 3-balls.
		Some of $B_1, . . . , B_l$ contains $S_i$ for $i=1, . . . , n$.
		Suppose that $\partial B_i$ contains $S_i$ for $i=1, . . . , n$. 
		There are mutually disjoint arcs $\{\alpha_1, . . . , \alpha_l\}$ properly embedded in $F_{12}-N(F^i_4)$ so that $\alpha_i$ connect $F$ and $B_i'$ for $i=n+1, . . . , l$ where $F$ is a one of the component of $F^i_{3}$ (See Figure \ref{Fig_4}).
		\begin{figure}[h]
					\centering
					\includegraphics[scale=0.9]{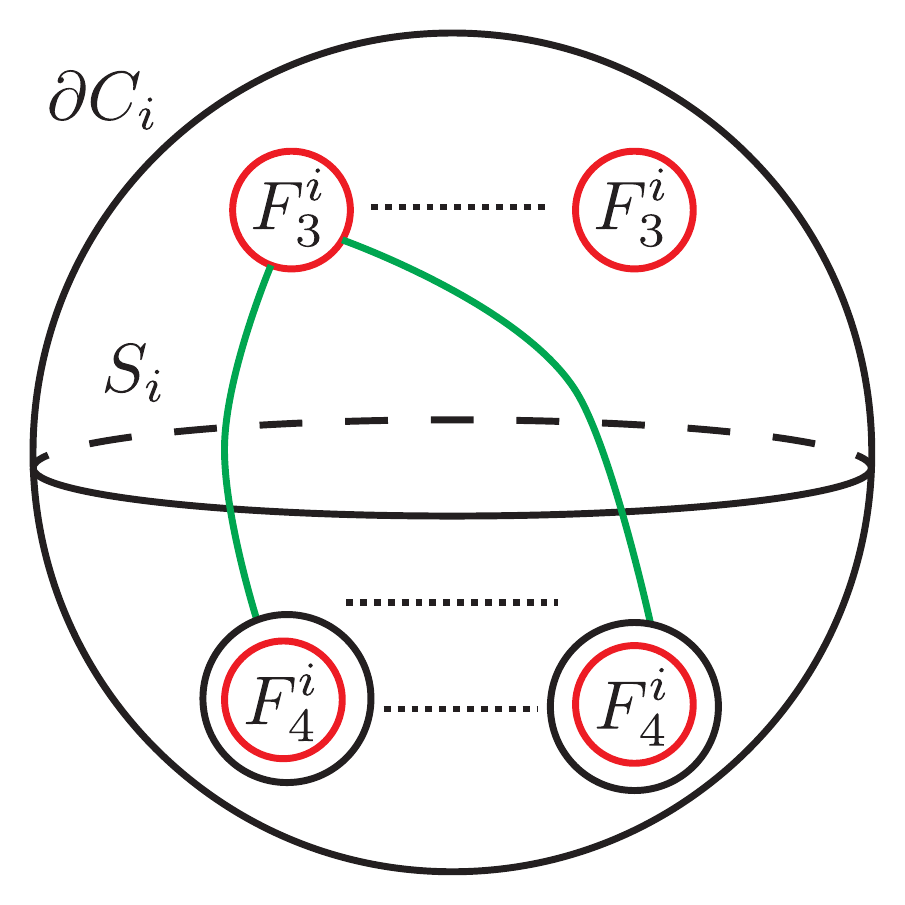}
					\caption
					{We take arcs connects $F^i_{3}$ and $F^i_4$. Such arcs connect $H_3'$ and $B_i$ for $i\neq 1, . . . , n$.}
					\label{Fig_4}
		\end{figure}
		Let $N'$ be a regular neighborhood of such arcs.
		Also there is an essential arc properly embedded in $N(F^i_{4})-F^i_{4}$ for $i=1, . . . , n$.
		Let $N''$ be a regular neighborhood of such arcs.
		$H_3''=H_3'\cup N'\cup ((B_1-N'')\cup\cdots\cup (B_l-N''))$, $H_2''=H_2'$ and $H_4''=H_4\cup N''\cup((C_1-N-N')\cup\cdots \cup (C_n-N-N'))$
		Since each of $H_3'\cap N'$ and $B_i\cap N'$ is a disk for $i=1, . . . , l$, $H_3''\cong H_3$.
		Also since each of $H_4'\cap N''$ and $(C_i-N-N')\cap N''$ is a disk for $i=1, . . . , n$, $H_4''\cong H_4$.
		Since $H_i''$ has no self intersection for $i=2, 3, 4$,  $H_2''\cup H_3''\cup H_4''$ is a type-$(g_2, g_3, g_4)$ decomposition.
	\end{proof}
	
	By Proposition \ref{prop1}, we obtain the following Proposition.

	\begin{prop}\label{prop_1}
	Let $M$ be a closed, connected, orientable 3-manifold. Then the following is satisfied.
		\begin{enumerate}
			\item[(1)] If $M$ has a type-$(0, 0, 0, 0)$ decomposition, then $M$ has a type-$(0, 0, 0)$ decomposition
			\item[(2)] If $M$ has a type-$(0, 0, 0, 1)$ decomposition, then $M$ has a type-$(0, 0, 1)$ decomposition.
			\item[(3)] If $M$ has a type-$(0, 0, 1, 1)$ decomposition, then $M$ has a type-$(0, 1, 1)$ decomposition.
			\item[(4)] If $M$ has a type-$(0, 1, 1, 1)$ decomposition, then $M$ has a type-$(1, 1, 1)$ decomposition.
		\end{enumerate}
	\end{prop}
	\setcounter{section}{5}
	Gomez-Larra\~naga studied handlebody decompositions in \cite{G}. He characterized 3-manifolds which admit decompositions with handlebodies of genera at most one.  Let $\mathbb B$ be a connected sum of a finite number of  $S^2\times S^1$'s,  let $\mathbb L$ and $\mathbb L_i$ be lens spaces, and let $\mathbb S(3)$ be a Seifert manifold  with at most three exceptional fibers. He showed the following proposition in \cite{G}.
	
\begin{prop}[\cite{G}]\label{GTh}Let $M$ be a closed, connected, orientable 3-manifold.
\begin{enumerate}

\item[(1)]
$M$ has a type-$(0,0,0)$ decomposition if and only if $M$ is homeomorphic to $\mathbb B$.
\item[(2)]
$M$ has a type-$(0,0,1)$ decomposition if and only if $M$ is homeomorphic to $\mathbb B$ or $\mathbb B \# \mathbb L$.
\item[(3)]
$M$ has a type-$(0,1,1)$ decomposition if and only if $M$ is homeomorphic to $\mathbb B$ or $\mathbb B \# \mathbb L$ or $\mathbb B \# \mathbb L_{1} \# \mathbb L_{2}$.
\item[(4)]
$M$ has a type-$(1,1,1)$ decomposition if and only if $M$ is homeomorphic to $\mathbb B$ or $\mathbb B \# \mathbb L$ or $\mathbb B \# \mathbb L_{1} \# \mathbb L_{2}$ or $\mathbb B \# \mathbb L_{1} \# \mathbb L_{2} \# \mathbb L_{3} $ or $\mathbb B \# \mathbb S(3)$.
\end{enumerate}
\end{prop}

	By Proposition \ref{GTh} and Proposition \ref{prop_1}, we can obtain the following theorem.
	\setcounter{section}{1}
	
	\setcounter{thm}{1}
	\begin{thm}\label{thm2}
	Let $M$ be a closed, connected, orientable 3-manifold. Then the following is satisfied.
		\begin{enumerate}

\item[(1)]
$M$ has a type-$(0, 0, 0, 0)$ decomposition if and only if $M$ is homeomorphic to $\mathbb B$.
\item[(2)]
$M$ has a type-$(0, 0, 0, 1)$ decomposition if and only if $M$ is homeomorphic to $\mathbb B$ or $\mathbb B \# \mathbb L$.
\item[(3)]
$M$ has a type-$(0, 0, 1, 1)$ decomposition if and only if $M$ is homeomorphic to $\mathbb B$ or  $\mathbb B \# \mathbb L$ or $\mathbb B \# \mathbb L_{1} \# \mathbb L_{2}$.
\item[(4)]
$M$ has a type-$(0, 1, 1, 1)$ decomposition if and only if $M$ is homeomorphic to $\mathbb B$ or $\mathbb B \# \mathbb L$ or $\mathbb B \# \mathbb L_{1} \# \mathbb L_{2}$ or $\mathbb B \# \mathbb L_{1} \# \mathbb L_{2} \# \mathbb L_{3} $ or $\mathbb B \# \mathbb S(3)$.
\end{enumerate}
	\end{thm}

Next we characterize 3-manifolds with type-$(1, 1, 1, 1)$ decomposition.
\begin{thm}
Let $M$ be a closed, connected, orientable 3-manifold. 
Then $M$ has a type-$(1, 1, 1, 1)$ decomposition if and only if $M$ is homeomorphic to $\mathbb B$ or  $\mathbb B \# \mathbb L$ or $\mathbb B \# \mathbb L_{1} \# \mathbb L_{2}$ or $\mathbb B \# \mathbb L_{1} \# \mathbb L_{2} \# \mathbb L_{3} $ or  $\mathbb B \# \mathbb L_{1} \# \mathbb L_{2} \# \mathbb L_{3} \# \mathbb L_{4}$ or $\mathbb B \# \mathbb S(4)$.
\end{thm}
\setcounter{section}{5}
\begin{proof}
	Let $H_1\cup H_2\cup H_3\cup H_4$ be a type-$(1, 1, 1, 1)$ decomposition.
	If the number of branched loci of  a type-$(1, 1, 1, 1)$ decomposition is at most one, one of the handlebodies has self intersection.
	Hence the number of branched loci of  a type-$(1, 1, 1, 1)$ decomposition is at least two. 
	We perform XI  moves to deform degree four branched loci of this decomposition to tribranched loci.
	First, we consider the case where one of $F_{ij}$'s contains a disk component.
 	\begin{claim}\label{cl5}
		Let $D$ be a disk component of $F_{ij}$. 
		If $\partial D$ is inessential in $\partial H_k$ or $\partial H_l$ for $k\neq l$, then  $M\cong M'\#S^2\times S^1$ or $M\cong M'\#\mathbb{L}$ where $M'$ has either a type-$(1, 1, 1, 1)$ or a type-$(0, 0, 1, 1)$ decomposition respectively whose partition is tribranched surface. 	
	\end{claim}
	\begin{proof}[Proof of Claim]
		Let $D$ be a disk component of $F_{12}$ and $\partial D$ be inessential in $\partial H_3$.
		Suppose that  exactly one of $F_{2i}$ is an emptyset for $i=3, 4$ and $F_{12}$ is only a disk.
		Assume that $F_{24}$ is an emptyset.
		Since $\partial H_2=F_{12}\cup F_{23}$ and $F_{12}$ is a disk, $F_{23}$ is a punctured torus.
		Then we can take a meridian disk of $H_2$ whose boundary is contained in $F_{23}$ and its regular neighborhood $h$ in $H_2$.
		After attaching $h$ to $H_3$ as a 2-handle, we can obtain a punctured lens space $h\cup H_3$.
		Then $M=M'\# \mathbb{L}$ where $\mathbb{L}$ is a lens space which is obtained from $h\cup H_3$ by capping off.
		$M'=H_1\cup (H_2-h)\cup H_4\cup B$ is a type-$(0, 0, 1, 1)$ decomposition where $B$ is a 3-ball since $H_2-h$ is a 3-ball.
		If exactly one of the $F_{1i}$ is an emptyset for $i=3, 4$ and $F_{12}$ is a disk, we can show samely as above. 
		We only remains two cases. 
		One is a case where each of  $F_{1i}$ and $F_{2i}$ is not emptyset  for $i=3, 4$ and the other is a case where $F_{12}$ has at least two components.
		
		Next, we suppose that  each of $F_{1i}$ and $F_{2i}$ are not emptyset for $i=3, 4$ or $F_{12}$ has at least two components.
		Let $D'$ be a disk in $\partial H_3$ such that $\partial D=\partial D'$ and $S=D\cup D'$.
		If  each of $F_{1i}$ and $F_{2i}$ are not emptyset for $i=3, 4$, we can take an arc properly embedded in $H_1\cup H_2$ so that one of the endpoints of the arc is contained in a component of $F_{14}$ and the other is contained in a component $F_{24}$ and it intersects $D$ exactly once.
		Also we can take an arc properly embedded in $H_4$ so that the end points of the each arcs are the same.
		If $F_{12}$ has at least two components, we can take an arc properly embedded in $H_1$ so that one of the endpoints if the arc is contained in $D$ and the other is contained in a component of $F_{12}$.
		Also we can take an arc properly embedded in $H_2$ so that  the end points of the each arcs are the same.
		Hence $S=D\cup D'$ is a non-separating sphere in $M$. Hence $M\cong M'\# S^2\times S^1$.
		To show a claim, we will show $M'$ has a type-$(1, 1, 1, 1)$ decomposition.

		Let $N(D)$ be a regular neighborhood of $D$ in $H_1\cup H_2$ and $N(D) \cap H_1=D_1$ and $N(D)\cap H_2=D_2$.
		Then $D_1$ and $D_2$ are disks.
		
		We can take a properly embedded disk $D_1'$ in $H_3$ such that $\partial D_1'=\partial D_1$ and disk $D_2'$ in $\partial H_3$ such that $\partial D_2=\partial D_2'$.
		Hence $D_2'$ may contain components each of $F_{3i}$ for $i=1, 2, 4$.
		
		We can take a regular neighborhood $N(S)$ of $S$ so that  $S_1=D_1\cup D_1'$ and $S_2=D_2\cup D_2'$ where $\partial N(S)=S_1\cup S_2$. See Figure \ref{Fig_7}.
		
		\begin{figure}[h]
					\centering
					\includegraphics[scale=0.7]{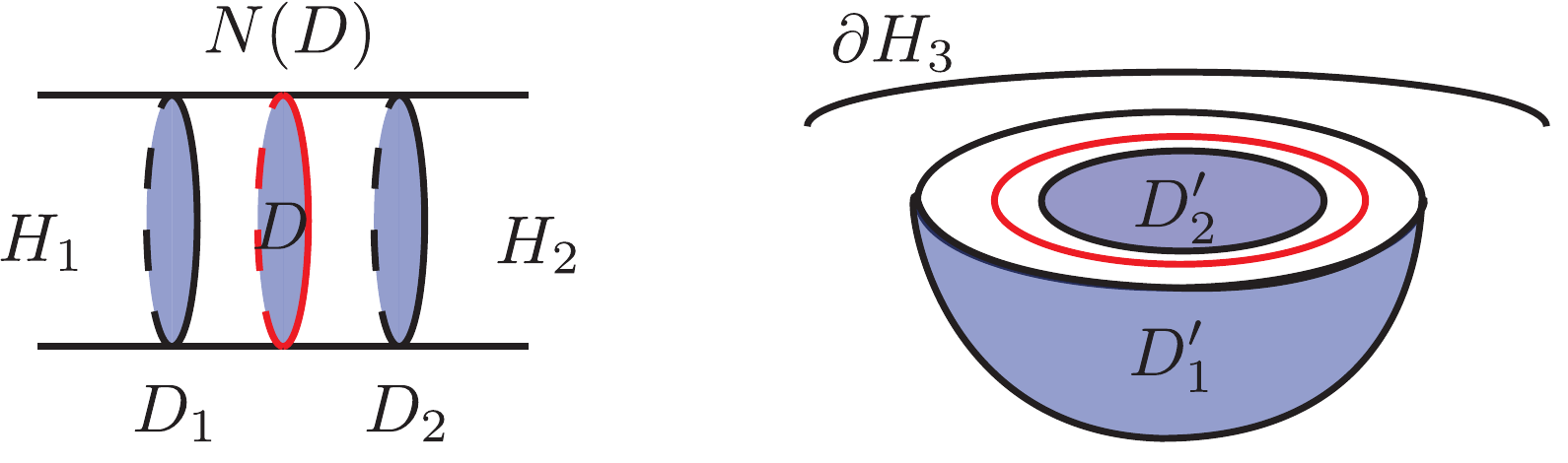}
					\caption
					{We take a regular neighborhood of $S$ so that $\partial N(S)=S_1\cup S_2$ satisfies that $S_1=D_1\cup D_1'$ and $S_2=D_2\cup D_2'$.}
					\label{Fig_7}
		\end{figure}
		Then $\partial (M-N(S))=S_1\cup S_2$.
		We cap off $M-N(S)$ by 3-balls $C_1$, $C_2$ with $\partial C_i=S_i$ for $i=1, 2$.
		Then $M'=M-N(S)\cup C_1\cup C_2$.
		Let $H_1'=(H_1-N(S))\cup C_1$, $H_2'=H_2-N(S)$, $H_4'=H_4$ and $H_3'=H_3-N(S)$.
		By definition, $M'=H_1'\cup H_2'\cup H_3'\cup H_4' \cup C_2$.
		 We note that $C_2$ may have  intersections with each of $H_i'$ for $i=1, 2, 4$.
		 
		 Let $F_{i}=C_2\cap H_i'$ for $i=1, 2, 4$. Recall that $F_2$ contains a disk $D_2$.
		 Suppose that one of the $F_{i}=\emptyset$ for $i=1, 4$.
		 We can assume that $F_4=\emptyset$.
		 Let $F$ be a component of $F_1$.
		 Then there is a components of $F_{12}$ whose boundary contains one of the components of $\partial F$.
		 We note that such components $F'$ of $F_{12}$ is adjacent to $F_{i4}$ or $F_{i3}$ for $i=1, 2$.
		 Suppose that $F'$ is adjacent to  $F_{14}$.
		 Then we can take an arc $\alpha$ properly embedded the component of $F_{12}$ which connects $C_2$ and $H_4$.
		 Then we define $H_4''=H_4\cup N(\alpha)\cup C_2$  and $H_i''=H_i'$ for $i= 1, 2, 3$.
		 Since $\partial C_2$ does not have an intersection with $H_4'$, each of $H_i''$ does not have self intersection.
		 Then $H_1''\cup H_2''\cup H_3''\cup H_4''$ is a type-$(1, 1, 1, 1)$ decomposition.
		 
		 Hence we can assume that $F_i\neq \emptyset $ for $i=1, 2, 4$.
		 We can take a regular neighborhood $N(F_4)$ of $F_4$ which contains no branched loci other than $\partial F_{4}$ and properly embedded disks in $C_2$ whose boundaries are components  of $\partial N(F_4)$.
		 Such disks cut open $C_2$ into some 3-balls.
		 Let $B_1, . . . , B_n$ be such 3-balls whose boundaries  contain  components of $F_4$.
		 On the other hands, let $B_1', . . . , B_m'$ be such 3-balls whose boundaries does not contain components of $F_4$.
		 
		There exists a properly embedded essential arcs in each components of $N(F_{4})-F_{4}$.
		Let $\alpha_1, . . . , \alpha_m$ be a subset of such arcs such that $\alpha_i$ connects $H_4'$ and $B_i'$.
		Let $N$ be a regular neighborhood of arcs $\alpha_1, . . . , \alpha_m$ and $H_4''=H_4\cup N\cup B_1'\cup\cdots\cup B_m'$, $H_i''=H_i'-N$ for $i=1, 2, 3$. 
		It is clear that $H_i''\cong H_i$ for $i=1, 2, 3$.
		Since each of components of $N\cap H_4$ is a disk and $N\cap B_i$ is a disk for $i=1, . . . , m$, $H_4''\cong H_4$.
		Since each of $\alpha_1, . . . , \alpha_m$ does not intersects branched loci, the intersection of handlebodies is a tribranched surface.

		Suppose that one of $\partial B_i$'s does not contain $F_1$.
		Let $\partial B_1$ does not  contain a component of $F_1$.
		Then $\partial B_1$ consists of $F_1$ and $F_4$.
		There is a components of $F_{14}$ which is adjacent to a component of  $F_1$.
		Let $F''$ be a such component of $F_{14}$.
		$F''$ adjacent to a components of either $F_{12}$ or $F_{13}$.
		Suppose that $F''$ adjacent to $F_{13}$.
		Then we can take an arc $\beta$ properly embedded $F''$ which connects $B_1$ and $H_3$.
		We replace $H_3'$ by $H_3'\cup N(\beta)\cup B_1$ where $N(\beta)$ is a regular neighborhood of $\beta$.
		Since each of  $H_3'\cap N(\beta)$ and $N(\beta)\cap B_1$ is a disk, a homeomorphism type of $H_3'$ does not change by this operation.
		This operation sends a tribranched surface to a tribranched  surface (See Figure \ref{Fig_5}).
		\begin{figure}[h]
					\centering
					\includegraphics[scale=0.6]{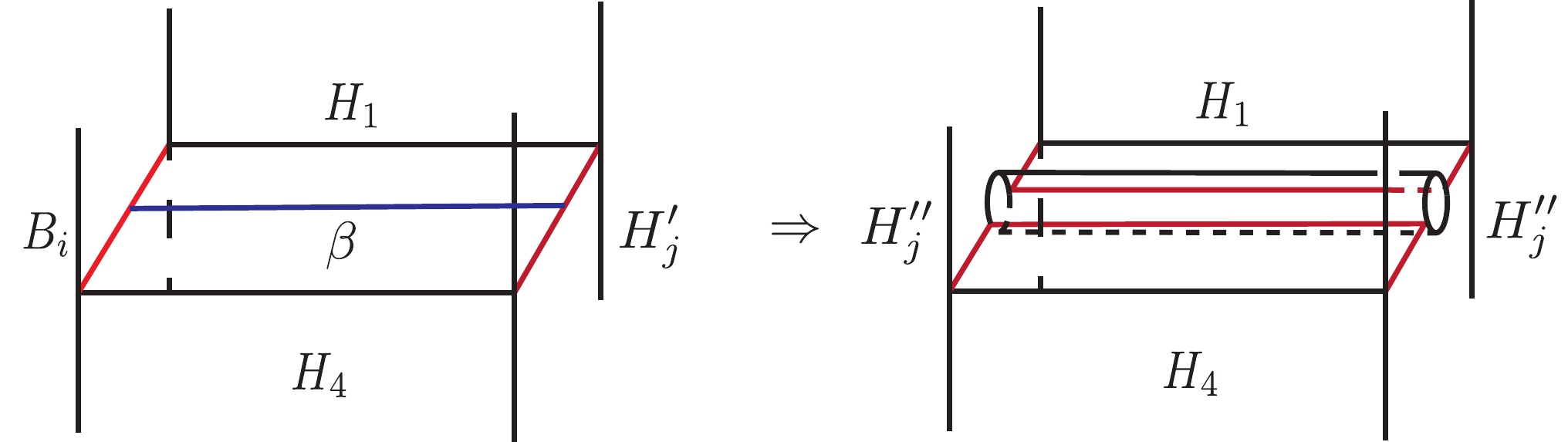}
					\caption
					{An operation which sends a tribranched surface to a tribranched surface}
					\label{Fig_5}
		\end{figure}
		Hence we can assume that each of $\partial B_i$'s satisfies that $F_j$ in $\partial B_i$ are not emptyset for  $i=1, . . . , n$, $j=1, 2, 4$.

		Let $N(F_1)$ be a regular neighborhood of $F_1$ in $\partial B_i$'s which contains no branched loci other than $\partial F_{1}$ for $i=1, . . . , n$.
		The components of $\partial N(F_1)$ in $F_4$ bounds disks properly embedded in $B_i$ for $i=1, . . . , n$.
		Such disks cuts open $B_i$ into some 3-balls for $i=1, . . . , n$.
		Let $B_1'', . . . , B_l''$ be such 3-balls whose boundary contain a component of $F_1$.
		We can take essential arcs properly embedded in $F_4-N(F_1)$ so that one of the endpoints is contained in $B_i''$ for $i=1, . . . , l$ the other is contained in $H_2$ since $F_j$ in $\partial B_i$ are not emptyset for  $i=1, . . . , n$, $j=1, 2, 4$.
		Let $\beta_1, . . . , \beta_l$ be such arcs such that $\beta_i$ connects $H_2$ and $B_i''$ for $i=1, . . . , l$ and $N'$ a regular neighborhood of union of such arcs.
		
		Also,  we can take essential arcs properly embedded in annuli $N(F_1)\cap F_4$.
		Let $\gamma_1, . . . , \gamma_n$ be a such arcs such that $\gamma_i$ connects $H_1$ and $B_i$ for $i=1, . . . , n$ and $N''$ be a regular neighborhood of union of such arcs.
		
		Let $H_1'''=H_1''\cap N''\cup B_1\cup\cdots B_n$, $H_2'''=H_2''\cup N'\cup B_1''\cup\cdots\cup B_l''$, $H_3'''=H_3''$
		and $H_4'''=H_4''$.
		Then $M'=H_1'''\cup H_2'''\cup H_3'''\cup H_4'''$.
		Since $N''\cap H_1''$ is a union of disks and $N''\cap B_i$ is a disk for $i=1, . . . ,n$, $H_1'''\cong H_1$.
		Also, Since $N'\cap H_2''$ is union of disks and $N'\cap B_i''$ is a disk for $i=1, . . . ,l$, $H_2'''\cong H_2$.
		Since each of $\beta_1, . . . , \beta_l$ and $\gamma_1, . . . , \gamma_n$ does not intersect the branched loci, the intersection of handlebodies is a tri-branched surface.
		Hence Then $M'=H_1'''\cup H_2'''\cup H_3'''\cup H_4'''$ is a type-$(1, 1, 1, 1)$ handlebody decomposition with tri-branched surface.
\end{proof}
	By Claim \ref{cl5}, we can assume that disk components of $F_{ij}$ is essential in $\partial H_k$ for $k\neq i, j$.
	If $F_{12}$ has a disk component $D$, we can take a regular neighborhood $N(D)$ of $D$ in $H_1\cup H_2$.
	If $\partial D$ is in $\partial H_4$, $H_4\cup N(D)$ is a punctured lens space $L$ since $\partial D$ is essential in $\partial H_3$.
	Then $M\cong Cap(L)\# M'$ where $M'$ is a capping off of $(H_1\cup H_2\cup H_3)-N(D)$.
	Hence $M'=H_1\cup H_2\cup H_3 \cup H_4'$ is a type-$(0, 1, 1, 1)$ decomposition where $H_4'$ is a 3-ball.
	This implies that $M\cong \mathbb{L}\#M'$ where $M'$ has a  type-$(0, 1, 1, 1)$ decomposition.
	
	If $F_{ij}$ has no disk component, $F_{ij}$ is annuli for $\{i, j\}\subset \{1, 2, 3, 4\}$.
	This implies that $M$ is a Seifert manifold with at most four singular fibers.
\end{proof}

We can see the difference between a multibranched handlebody decomposition and a handlebody decomposition by Theorem \ref{thm2}.
We can show that any orientable, closed 3-manifold admits a type-$(0, 0, 0, 0)$ handlebody decomposition. 
On the other hand, there are a lot of orientable, closed 3-manifold which does not admits type-$(0, 0, 0, 0)$ multibranched handlebody decomposition.

\section{Acknowledgement}
The author thanks to his supervisor Koya Shimokawa for very meaningful discussion and insight.
The author was partially supported by Grant-in-Aid for JSPS Research Fellow from JSPS KAKENHI Grant Number JP20J20545.

\end{document}